\documentclass[11pt]{amsart}
\usepackage[utf8x]{inputenc}
\usepackage{amsfonts}
\usepackage{amsthm}
\usepackage[section]{placeins}
\usepackage{float}
\usepackage{perpage}
\MakeSorted{figure}
\MakeSorted{table}
\usepackage{latexsym}
\usepackage[dvips]{graphicx}
\usepackage{color}
\usepackage{amsmath,amssymb,graphics,amsthm}
\usepackage[english]{babel}

\usepackage{mathrsfs}
\usepackage{amsmath}
\usepackage{graphicx}
\usepackage[colorinlistoftodos]{todonotes}
\usepackage{enumerate} 
\usepackage{hyperref}
\usepackage{fullpage}

\newtheorem{thm}{Theorem}[section]

\newtheorem{cor}[thm]{Corollary}
\newtheorem{lema}[thm]{Lemma}

\newtheorem{prop}[thm]{Proposition}

\newcommand{\N}{{\mathbb N}}

\newcommand{\dom}{\text{\sf dom}}
\newcommand{\im}{\text{\sf im}}
\newcommand{\ev}{\text{\sf ev}}

\newcommand{\cantor}{2^{\N}}

\title{Topologies on the symmetric inverse semigroup}
\author{J.  P\'erez}
\address{Escuela de Matem\'aticas, Universidad Industrial de Santander, C.P. 680001, Bucaramanga - Colombia.}
\email{jersonenrique\_64@hotmail.com}

\author{C. Uzc\'{a}tegui}
\address{Escuela de Matem\'aticas, Universidad Industrial de Santander, C.P. 680001, Bucaramanga - Colombia.}
\email{cuzcatea@saber.uis.edu.co}

\date{\today}

\begin{document}

\begin{abstract}
The symmetric inverse semigroup $I(X)$ on a set $X$ is the collection of all partial bijections between subsets of $X$ with composition as the algebraic operation. We study a minimal Hausdorff inverse semigroup topologies on $I(X)$. When $X$ is countable, we show some Polish semigroup topologies on $I(X)$. 

\end{abstract}


\maketitle

\section{Introduction}

The symmetric inverse semigroup $I(X)$ on a set $X$ is the collection of all partial bijections between subsets of $X$ with composition as the algebraic operation. Among inverse semigroups, $I(X)$  plays a role analogous to that played by the symmetric group $S_\infty(X)$ for groups: every inverse semigroup  $S$ is isomorphic to a subsemigroup of $I(S)$. In this paper we study some semigroup topologies on $I(X)$. In the case of a countable $X$, the topologies on $I(X)$ can be found  Polish (i.e. completely metrizable and separable). 

The symmetric group $S_\infty(X)$ is a subsemigroup  of $X^X$ and also of $I(X)$. The usual topology of $S_\infty(X)$ is the one it inherits from the product topology on $X^X$. We define a topology $\tau_{pp}$ on $I(X)$, we call it the {\em partial product topology}, which also induces in $S_\infty(X)$ its usual product topology. It turns out that $\tau_{pp}$ is the minimal inverse semigroup Hausdorff topology on $I(X)$. We present a characterization of $\tau_{pp}$  analogous to that well known fact that  the product topology is the smallest topology making all projections continuous. We show that there is an onto map $\pi: S_\infty(Y) \to I(X)$ 
such that $\tau_{pp}$ is the quotient topology given by $\pi$. Moreover, $\tau_{pp}$ is the unique inverse semigroup Hausdorff topology on $I(X)$ with respect to which $\pi$ is continuous. The topology  $\tau_{pp}$  was independently studied in \cite{elliott2019,elliott2020}. Some of the results we present are coming  from \cite{TesisJerson}.

Recently there has been a lot of interest about
groups and semigroups which admits a unique Polish  group (semigroup) topology. A theorem of Kallman \cite{kallman1979} says that $S_\infty(\N)$ has a unique Polish group topology, namely the product topology (see \cite[Exercise 2.3.9]{GA}). The same happens with the group of homeomorphisms of the Cantor space as shown by Rosendal and Solecki \cite{RosendalSolecki2007}.   Mesyan et al. \cite{MES} showed that the product topology is the unique Polish semigroup topology on $\N^\N$. 
More recently, Elliott et. al made  a very extensive study of Polish semigroup topologies \cite{elliott2019, elliott2020}. In particular,  they showed that there is a unique Polish inverse semigroup topology on $I(\N)$, namely  $\tau_{pp}$.  Their proof is based on a general criteria for getting automatic continuity for $I(\N)$ using the fact that $S_\infty(\N)$ admits an unique Polish group topology. Our approach is quite different but we obtained a  weaker result, since it needs some  extra condition (see Theorem \ref{unicidad2}).

\section{Preliminaries}

A {\em semigroup} is a non-empty set $S$ together an associative binary operation $\circ$. To simplify the notation we sometimes write $st$ in placed of $s\circ t$.
Let $S$ be a semigroup. We say that $S$ is a regular  if for all $s\in S$, there is $t\in S$ such that $st s=s$ and $tst=t$. In this case we say that $t$ is an inverse of $s$. If each element have a unique inverse we call $S$ an inverse semigroup. We denote by  $s^\ast$  the inverse of $s$.

Let $S$ be a semigroup and $\tau$ a topology on $S$. If the multiplication $S\times S\to S$ is continuous, we call $S$ a topological semigroup.  An inverse semigroup $S$ is called topological, if it is a topological semigroup and the function $i:S\to S$, $s\to s^\ast$ is continuous.

\begin{prop}\label{iseq}
Let $S$ be a inverse semigroup and $\tau$ a topology on $S$ such that $(S,\tau)$ is a Hausdorff topological semigroup.
Suppose $(s_n)_n$ is a sequence on $S$ such that $s_n\to s$ and $s_n^{\ast}\to t$. Then $t=s^*$.
\end{prop}

\begin{proof} Notice that $s_n=s_ns_n^{\ast}s_n$ and $s_n^{\ast}=s_n^{\ast}s_ns_n^{\ast}$ for all $n\in \N$. By the continuity of the semigroup operation, we have that  $s_n=s_ns_n^{\ast}s_n\to sts$ and $s_n^{\ast}=s_n^{\ast}s_ns_n^{\ast}\to tst$. Thus $s=sts$ and $t=tst$, in other words, $t=s^{\ast}$.
\end{proof}

Let $S$ be an inverse semigroup and $\rho$ a metric on $S$. Define another metric as follow
\[
\rho^\ast(s,t)=\rho(s^\ast,t^\ast).
\]

\begin{prop}\label{i_iso}
Let $S$ be a inverse semigroup and $\rho$ a  metric on $S$.
\begin{itemize}
\item[(i)] $i:(S,\tau_\rho)\to (S,\tau_{\rho^\ast})$, $s\mapsto s^\ast$, is an isometry.  Moreover, if $\rho$ is complete, so is $\rho^\ast$. 

\item[(ii)] Suppose  $(S,\tau_\rho)$ is a topological semigroup, then $(S,\tau_{\rho^\ast})$ is a topological semigroup. 

\end{itemize}
\end{prop}

\begin{proof}
(i) It is straightforward. 

(ii) Define $g:(S,\rho^*)\times (S,\rho^*)\to S\times S$ such that $(s,t)\to (t,s)$. Note that $g$ is continuous function on $\tau_{\rho^\ast}$ and $\ast=i\circ \ast \circ (i\times i) \circ g$, therefore $\ast$ is a continuous function on $\tau_{\rho^\ast}$. And $\rho^\ast$ is complete since $\rho$ is complete and $i$ is an isometric homeomorphism by Proposition \ref{i_iso}(i).
\end{proof}

\vspace{0.5cm}

\begin{thm}\label{metrica_suma} 
Let $S$ be an inverse semigroup and $\rho$ be a complete metric on $S$ such that $(S,\tau_\rho)$ is a topological semigroup. Let $d=\rho+\rho^*$,  then $d$ is a complete metric on $S$ and $(S,\tau_d)$ is a topological inverse semigroup. 
\end{thm}

\begin{proof}
Let us see that $d$ is complete. Let $(s_n)_n$ be a $d$-Cauchy sequence and $\epsilon>0$, then there is $N\in \mathbb{N}$ such that if $n,m>N$, $d(s_n,s_m)<\epsilon$. Thus $\rho(s_n,s_m)<\epsilon$ and $\rho^\ast(s_n,s_m)<\epsilon$ for all $n,m>N$. Since $\rho$ and $\rho^\ast$ are complete, there are $s,t \in S$ such that $s_n\stackrel{\rho}{\longrightarrow} s$  and $s_n\stackrel{\rho^\ast}{\longrightarrow} t$. By Proposition \ref{iseq}, $t=s$. It is easy to see that $s_n\stackrel{d}{\longrightarrow} s$. 

Now we show that $\tau_d$ is an inverse semigroup topology. First we note that  $i^{-1}(B_{d}(s,r))=B_{d}(s^\ast,r)$, thus  $i$ is continuous. To see that the operation $(S,\tau_d)\times(S,\tau_d)\to(S,\tau_d)$ is continuous, let $(x_n)_n$ and $(y_n)_n$ be sequences such that $x_n\stackrel{d}\to x$ and $y_n\stackrel{d}\to y$, we show that $x_n y_n\stackrel{d}\to x y$. Clearly we have $x_n\stackrel{\rho}\to x$, $y_n\stackrel{\rho}\to y$,  $x_n^\ast\stackrel{\rho}\to x^\ast$ and $y_n^\ast\stackrel{\rho}\to y^\ast$. Since the operation is $\tau_\rho$-continuous, $x_n y_n\stackrel{\rho}\to x y$ and $(x_n y_n)^\ast\stackrel{\rho}\to (x y)^\ast$, therefore $x_n y_n\stackrel{d}\to x y$.
\end{proof}

\vspace{0.5cm}

\section{The partial product topology}

The symmetric inverse semigroup on a set $X$  is defined as follows:
$$
I(X)=\{f:A\rightarrow B\text{ }|\text{ }A,B\subseteq X\text{ and }f \text{ is bijective}\}.
$$
For  $f:A\rightarrow B$  in $I(X)$ we denote $A=\dom(f)$ and $B=\im(f)$. Let 
\[
D_x=\{f\in I(X):\; x\in \dom (f)\}
\]
and $2^X$ denotes the power set of $X$. Let $S_\infty(X)$ be the  symmetric group, that is, the collection of all bijection from $X$ to $X$. The following functions play an  analogous role as the proyection functions  for the product topology. 
\begin{enumerate}
\item[] $\dom: I(X)\to 2^X$, $f\mapsto \dom(f)$,
    
\item[] $\im:  I(X)\to 2^X$, $f \mapsto \im(f)$,
    
\item[] $\ev_{x}: D_{x}\to  X$, $f\mapsto f(x)$ for all $x\in X$.
\end{enumerate}

\medskip

The operation on $I(X)$ is the usual composition, namely, given  $f,g\in I(X)$, then  $f\circ g$ is defined by letting $\dom(f\circ g)=g^{-1}(\dom(f)\cap \im(g))$ and if $x\in \dom(f\circ g)$ then $(f\circ g)(x)=f(g(x))$.
The idempotents of $I(X)$ are the partial identities $1_A:A\to A$, $1_A(x)=x$ for all $x\in A$ and  $A\subseteq X$. Notice that $1_\emptyset$ is the empty function which also belongs to $I(X)$.
For $x,y\in X$, let 
$$
\begin{array}{rcl}
v(x,y) &= & \{f\in I(X)\text{ }|\text{ }x\in \dom(f)\text{ and }f(x)=y\},\\
w_{1}(x) &= & \{f\in I(X)\text{ }|\text{ }x\not\in \dom(f)\},\\
w_{2}(y) &= & \{f\in I(X)\text{ }|\text{ }y\not\in \im(f)\}.
\end{array}
$$
It is clear that sets  $v(x,y)$ are motivated by the usual subbase for the  product topology on $X^X$.   
As we will see, these sets are tightly  related to any $T_1$ semigroup topology on $I(X)$. 

\begin{lema}
\label{opercontinuas}
Let $c:I(X)\times I(X)\to I(X)$ be given by $c(f,g)=f\circ g$ and $i:I(X)\to I(X)$ be given by $i(f)=f^{-1}$. Then, for all $x,y\in X$, we have

\begin{itemize}
\item[(i)] $c^{-1}(v(x,y))=\displaystyle\bigcup_{z\in X}(v(z,y)\times v(x,z)$).
\item[(ii)] $c^{-1}(w_{1}(x))=(I(X)\times w_{1}(x))\cup \displaystyle\bigcup_{z\in X} (w_{1}(z)\times v(x,z))$.

\item[(iii)] $c^{-1}(w_{2}(y))=(w_{2}(y)\times (I(X)) \cup \displaystyle\bigcup_{z\in X} (v(z,y)\times w_{2}(z))$.

\item[(iv)] $i^{-1}(v(x,y))=v(y,x)$,

\item[(v)]$i^{-1}(w_{1}(x))=w_{2}(x)$,

\item[(vi)]$i^{-1}(w_{2}(y))=w_{1}(y)$.
\end{itemize}
\end{lema}

\begin{proof}
It is straightforward. 
\end{proof}

Now we introduce some  semigroup topologies on $I(X)$.  Let $\tau_0$ be the topology generated by $\{ v(x,y): x, y\in X\}$, $\tau_1$ be  the topology generated by   $\{v(x,y), w_1(x): x,y\in X\}$,  $\tau_2$ be the topology generated by  $\{v(x,y), w_2(y): x,y\in X\}$ and $\tau_{pp}$ be the topology generated by $\{v(x,y), w_1(x), w_2(y): x,y\in X\}$. These topologies were independently defined in \cite{elliott2019,elliott2020} and show essentially the following. 

\begin{thm}
\label{topI(X)}
\begin{itemize}
\item[(i)]  $(I(X), \tau_0)$ is a $T_0$ topological semigroup but it is not $T_1$. 

\item[(ii)] $(I(X), \tau_1)$  is a Hausdorff topological semigroup.

\item[(iii)]  $(I(X),\tau_2)$  is  a Hausdorff topological semigroup and $i:(I(X), \tau_1)\to (I(X),\tau_2)$ is a homeomorphism and an antisomorphism. 

\item[(iv)] $(I(X), \tau_{pp})$  is a Hausdorff topological inverse semigroup.

\end{itemize}
\end{thm}

\begin{proof} Let $\tau$ be any of the topologies mentioned in the hypothesis. From Lemma \ref{opercontinuas} it follows that  $\tau$  is a semigroup topology on $I(X)$.

(i)  It remains to show that the $\tau_0$ is $T_0$. In fact, let  $f, g\in I(X)$ with $f\neq g$. There are several cases to consider, we treat only one,  the others are analogous. Assume there is $x\in \dom(f)\cap \dom(g)$ such that $f(x)\neq g(x)$. Then $f\in v(x, f(x)))$ and $g\not\in v(x,f(x))$. Notice that the only $\tau_0$-open set that contains $1_\emptyset$ is $I(X)$, therefore $(I(X),\tau_0)$ is not $T_1$.

(ii) It remains to show that $\tau_1$ is $T_2$.  This is done easily by analyzing the following cases. Let $f\neq g$ in $I(X)$. (a) There is $x\in \dom(f)\triangle \dom(g)$ (b) There is $x\in \dom(f)\cap \dom(g)$ with $f(x)\neq g(x)$.

(iii)  It follows from Lemma \ref{opercontinuas}. 
(iv)It follows from Lemma \ref{opercontinuas}. 
\end{proof}

We call $\tau_{pp}$ the {\em partial product topology}. The following result shows that $\tau_{pp}$ is minimal among all Hausdorff inverse semigroup topologies on $I(X)$.

\begin{thm}
\label{tppminima}
Suppose $(I(X), \tau)$ is a $T_{1}$ topological semigroup. Then 
\begin{itemize}
\item[(i)] Each $v(x,y)$ is clopen and  $w_1(y)$ and $w_2(y)$ are closed for all $y$.  

\item[(ii)] Suppose $\tau$ is an inverse semigroup topology. The relation $\subseteq$ is closed  on $I(X)\times I(X)$ iff  every $w_1(x)$ is open. 

\item[(iii)] If $\tau$ is an inverse semigroup Hausdorff topology, then $\tau_{pp}\subseteq \tau$. 
\end{itemize}

\end{thm}

\begin{proof}
(i) For each $x,y\in X$, let  $u_{x,y}\in I(X)$ be such that $\dom(u_{x,y})=\{x\}$ and $\im(u_{x,y})=\{y\}$.  Let $\varphi:I(X)\to I(X)$ given by $\varphi(h)=u_{y,x}\circ h \circ u_{y,x}$. Then $\varphi$ is continuous and 
\[
I(X)\setminus v(x,y)=\varphi^{-1}(1_\emptyset).
\]
Thus $v(x,y)$ is open. We also have that $
v(x,y)=\varphi^{-1}(u_{y,x})$.
Thus $v(x,y)$ is also closed. To see that $w_1(x)$ and $w_2(y)$ are  closed observe that $
I(X)\setminus w_1(x)=\bigcup_{z\in X}v(x,z)$ and $I(X)\setminus w_2(y)=\bigcup_{z\in X} v(z,y)$. 

(ii) Suppose $\subseteq$ is closed. Then $f\not\in w_1(x)$ iff $u_{x,x}\subseteq f^{-1}\circ f$. Thus the complement of $w_1(x)$ is closed.  Conversely, $f\not\subseteq g$ iff for some $ x,y\in X$ we have  that $f\in v(x,y)$ and  $g\in ((I(X)\setminus v(x,y))\cup w_1(x)$.  From this  follows that the complement of $\subseteq$ is open. 

(iii) Suppose that $\tau$ is an inverse semigroup Hausdorff topology on $I(X)$. From (i) we have that each $v(x,y)$ is $\tau$-open. To see that each $w_1(x)$ is $\tau$-open, by (ii), it suffices to show that $\subseteq$ is $\tau\times\tau$-closed. In fact,  notice that $f\subseteq g$ iff $f=f\circ f^{-1}\circ g$. Since $\circ$ is $\tau$-continuous, then $\subseteq$ is closed. Finally, since $\tau$ is an inverse semigroup topology and each $w_1(x)$ is $\tau$-open,  from Lemma \ref{opercontinuas} we conclude that each $w_2(y)$ is also $\tau$-open.  Thus $\tau_{pp}\subseteq \tau$. 
\end{proof}

\begin{cor}
$(I(X),\tau_{pp})$ is a regular space.
\end{cor}

Next we present a characterization of when the collection of idempotent of $I(X)$ is compact. 

\begin{thm}
\label{compact-idempotent}
Let $\tau$ be an inverse semigroup Hausdorff topology on $I(X)$. The map $A\mapsto 1_A$ from $2^X$ to $I(X)$ is continuous iff the collection of idempotents is compact. 
\end{thm}

\begin{proof}
The collection of idempotent of $I(X)$ is $J=\{1_A:\: A\subseteq X\}$. Thus if the map $A\mapsto 1_A$ is continuous, then the collection of idempotent is compact. Conversely, suppose $J$ is compact. By Theorem \ref{proyectiva},  the function $\dom: (I(X),\tau_{pp})\to 2^X$ is continuous and $\tau_{pp}\subseteq \tau$ (by  Theorem \ref{tppminima}), thus $\dom$ is also continuous with respect to $\tau$. Thus $\dom|_J:(J, \tau)\to \cantor$ is a continuous bijection and hence an homeomorphism as $J$ is compact Hausdorff. Its inverse is the map $A\mapsto 1_A$.
\end{proof}

Next result shows that all three topologies $\tau_1$, $\tau_2$ and $\tau_{pp}$ are different.

\begin{prop}\label{28}
\begin{itemize}
\item[(i)] $w_{2}(y)$ is  $\tau_1$-nowhere dense for all $y\in X$.
\item[(ii)] $w_{1}(x)$ is  $\tau_2$-nowhere dense for all $x\in X$.
\end{itemize}
\end{prop}

\begin{proof}
(i) By Theorem \ref{topI(X)}, $\tau_1$ is a $T_1$  semigroup topology and by Proposition \ref{tppminima} each $w_2(y)$ is $\tau_1$-closed. We show that each $w_2(y)$ has empty $\tau_1$-interior. 
Let $V$ be  the basic $\tau_1$-open set 
\[
V=\bigcap_{i=1}^n v(x_{i},y_{i})\cap \displaystyle\bigcap_{i=1}^m w_{1}(z_{i}).
\]
If $y=y_i$ for some $i$, then $V\cap w_2(y)=\emptyset$. Otherwise, pick $x\not\in\{x_1,\cdots, x_n, z_1, \cdots, z_m\}$. Then  $V\cap v(x,y)\neq\emptyset$ and $V\cap v(x,y)\cap w_2(y)=\emptyset$. 

(ii) It is proved analogously to (i). 

\end{proof}

\begin{prop}
Let $\tau$ be a semigroup topology on $I(X)$. If $w_1(x)$ is open for some $x$, then $w_1(y)$ is open for every $y$.  Moreover, if $\tau$ is an inverse semigroup topology on $I(X)$ and $w_1(x)$ is open for some $x$, then $w_2(y)$ is open for every $y$.
\end{prop}

\begin{proof} Let $g\in S_\infty(X)$ be such that $g(x)=y$. Let $\varphi:I(X)\to I(X)$ be given by $\varphi(f)=f\circ g$. Clearly $\varphi$ is a homeomorphism and  $ w_1(y)=\varphi^{-1}(w_1(x))$. The second claim follows from the first and Lemma \ref{opercontinuas}. 
\end{proof}

The next result is a generalization of the fact that the product topology is the smallest topology with respect to which  all projections are continuous. 

\begin{thm}
\label{proyectiva}
\begin{itemize}
\item[(i)] $\tau_{pp}$ is the smallest topology such that  $\dom:I(X)\to 2^X$, $\im:I(X)\to 2^X$ and $\ev_{x}:D_x\to X$ ($x\in X$) are continuous, where $2^X$ is endowed with the usual product topology and $X$ with the discrete topology.

\item[(ii)] $\dom$ and $\im$ are open maps when $I(X)$ is endowed with the topology $\tau_{pp}$.
    
\item[(iii)] Let $Y$ be a topological space and   $\varphi:Y\to (I(X),\tau_{pp})$ be a map. Then $\varphi$   is continuous if, and only if, $\dom\circ \varphi$, $\im\circ \varphi$ are continuous and $\ev_x\circ \varphi$ is  continuous on  $\varphi^{-1}(D_x)$ for all $x\in X$.

\item[(iv)] The collection of idempotent is $\tau_{pp}$-compact.
\end{itemize}
\end{thm}

\begin{proof}
(i) For each $x\in X$, let $v_{x}=\{A\subseteq X: \;x\in A\}$. Each $v_x$ is clopen in $2^X$ and they form a subbasis for $2^X$. The result follows immediately from the next identities:
\[
w_1(x)=(\dom)^{-1}(2^X\setminus v_{x}), \;\;w_2(x)=(\im)^{-1}(2^X\setminus v_{x}),\;\; v(x,y)=\ev_{x}^{-1}(\{y\}).
\]

(ii) Let $\{x_i:\; 1\leq i\leq n\}$, $\{y_i: \; 1\leq i\leq n\}$, $\{u_j:\; 1\leq j\leq m\}$ and $\{z_k\; 1\leq k\leq l\}$  be finite subsets of $\N$ and consider the basic open set in $I(\N)$:
\[
V=\bigcap_{i=1}^n v(x_i,y_i)\cap \bigcap_{j=1}^m w_1(u_j)\cap \bigcap_{k=1}^l w_2(z_k).
\]
Then 
\[
\{\dom(f):\; f\in V\}=\{A\in 2^X:\; \forall i\leq n (x_i\in A) \;\;\text{and}\;\; \forall j\leq m  ( u_j\not\in A)\}. 
\]
is clearly open. Analogously 
\[
\{\im(f):\; f\in V\}=\{B\in 2^X:\; \forall i\leq n (y_i\in B) \;\;\text{and}\;\; \forall k\leq l  ( z_k\not\in B)\}. 
\]

(iii) Let $\varphi:Y\to I(X)$ satisfying the hypothesis we will show that $\varphi$ is continuous. Let $V\subseteq I(X)$ be a basic open set as in (ii). Let $y\in \varphi^{-1}(V)$, then $x_i\in \dom (\varphi(y))$ for $1\leq i\leq n$. By (ii)  $\dom(V)$  and $\im(V)$ are open in $2^X$. We claim that   
\[
y\in  (\dom\circ \varphi)^{-1}(\dom(V)) \cap (\im\circ  \varphi)^{-1}(\im(V))\cap\bigcap_{i=1}^n (\ev_{x_i}\circ \varphi)^{-1}(\{y_i\})\;\; \subseteq \varphi^{-1}(V).
\]
In fact, since $\dom(\varphi(y)) \in \dom(V)$, $u_i\not\in \dom(\varphi(y))$ for all $j\leq m$. Analogously, $z_k\not\in \im (\varphi(y))$ for all $k\leq l$.  Finally, as $y \in  (\ev_{x_i}\circ \varphi)^{-1}(\{y_i\})$, then $\varphi(y)(x_i)=y_i$. Thus $\varphi(y)\in V$. We have shown that $\varphi^{-1}(V)$ is open and hence $\varphi$ is continuous.

(iv) This follows from Theorem \ref{compact-idempotent}. In fact, let $\varphi:2^X\to (I(X), \tau_{pp})$ given by $\varphi(A)=1_A$. We claim that $\varphi$ is continuous.  Notice that $(\dom \circ \varphi)(A)=(\im\circ\varphi) (A)= A$ for all $A$. And $(\ev_x\circ \varphi)(A)=x$ for all $A$ with $x\in A$. Thus,  by (iii),  $\varphi$ is continuous. 

\end{proof}

Now we characterize convergence of nets with respect to the topologies defined above. 

\begin{thm}\label{conv}
Let $(f_{\lambda})_{\lambda\in \Lambda}$  be a net in $I(X)$ and $f\in I(X)$. Then, $f_{\lambda}\stackrel{\tau_1}\to f$ if and only if for all $x\in X$ 

\begin{itemize}
\item[(i)] If $x\in \dom(f)$, then there is $\lambda_{0}\in \Lambda$ such that if $\lambda\geq \lambda_{0}$ then $x\in \dom(f_{\lambda})$ and $f_{\lambda}(x)=f(x)$.

\item[(ii)] If $x\not\in \dom(f)$, then there is $\lambda_{0}\in \Lambda$ such that if $\lambda\geq \lambda_{0}$ then $x\not\in \dom(f_{\lambda})$.
\end{itemize}
\end{thm}

\begin{proof}
It is straightforward from the definition of $\tau_1$.
\end{proof}

\begin{thm}\label{conv2}
Let $(f_{\lambda})_{\lambda\in \Lambda}$ be a net in $I(X)$ and $f\in I(X)$. Then,
\begin{itemize}
\item[(i)]
$f_{\lambda}\stackrel{\tau_1}\to f$ if and only $f_{\lambda}^{-1}\stackrel{\tau_2}\to f^{-1}$.

\item[(ii)] $f_{\lambda}\stackrel{\tau_{pp}}\to f$ if and only $f_{\lambda}\stackrel{\tau_1}\to f$ and $f_{\lambda}\stackrel{\tau_2}\to f$ if, and only if, $f_{\lambda}\stackrel{\tau_1}\to f$ and $f_{\lambda}^{-1}\stackrel{\tau_1}\to f^{-1}$. 

\end{itemize}

\end{thm}

\begin{proof}
(i) It follows immediately from the fact that $f\in w_1(y)$ iff $f^{-1}\in w_2(y)$. 

(ii) The first equivalence follows from the fact that $\tau_{pp}$ is the supreme of $\tau_1$ and $\tau_2$. The second one follows from (i). 
\end{proof}

\section{The case $X$ countable}

In this section we will study $I(X)$ when $X$ is a countable set.  We will show that $(I(\mathbb{N}), \tau_{pp})$ is Polish.  We will assume that $X=\N$. In this case $(I(\mathbb{N}), \tau_{pp})$ is a Hausdorff, regular and second-countable space, then by Urysohn theorem, $(I(\mathbb{N}), \tau_{pp})$ is metrizable. In order to define a metric compatible with the partial product topology we define some auxiliary functions.  Let $f,g \in I(\mathbb{N})$ and consider the function $a_{(f,g)}, b_{(f,g)}\in \cantor$ defined by 

\medskip

\begin{itemize}
\item[] 
$a_{(f,g)}(n)=\left\{ \begin{array}{ll}
             0 &   \text{if }   \; n \in (\dom(f)\cap \dom(g))\cup ((\dom(f))^c\cap (\dom(g))^c),  \\
             \\ 1 &  \text{otherwise}.
            \end{array}
 \right.$

\item[] 
$b_{(f,g)}(n)=\left\{ \begin{array}{ll}
             0 &   \text{if } \; n \not\in \dom(f)\cap \dom(g),  \\
         \\ min\{1,|f(n)-g(n)|\} & \text{if } \;  n \in \dom(f)\cap \dom(g). \
             
            \end{array}
   \right.$
\end{itemize}

\medskip

Now consider the following metric
\[
\rho(f,g)=\displaystyle\sum_{n\in \mathbb{N}}\frac{a_{(f,g)}(n)+b_{(f,g)}(n)}{2^n}.
\]

We show next that $(I(\N),\tau_\rho)$ is a Polish semigroup (but not a topological  inverse semigroup). Recall that we also have a metric $\rho^\ast$ on $I(\N)$ given by 
\[
\rho^\ast(f,g)=\rho(f^{-1},g^{-1}).
\]

\begin{prop}
\label{compa}
\begin{itemize}\item[(i)] $(I(\N),\tau_1)$ is metrizable and $\rho$ is a compatible metric for $\tau_1$.

\item[(ii)] $(I(\N), \tau_2)$ is  metrizable and  $\rho^\ast$ is a compatible metric for $\tau_2$.
\end{itemize}
\end{prop}

\begin{proof}
(i) It is easy to verify that $\rho$ is indeed a metric. 
Since $(I(\mathbb{N}),\tau_1)$ is metrizable, to show that $\rho$ is compatible with $\tau_1$ it suffices to show that they induce the same convergent sequences.

Let $(f_n)_{n\in \mathbb{N}}$ be a sequence and $f\in I(\mathbb{N})$ be such that $f_n\stackrel{\rho}\to f$.
We use Theorem \ref{conv} to show that $f_n\stackrel{\tau_1}\to f$.
Let $m\in \mathbb{N}$, there is $N\in \mathbb{N}$ such that if $n\geq N$, $\rho(f_n,f)<\frac{1}{2^m}$, therefore $a_{(f,f_{n})}(m)=0$ and $b_{(f,f_{n})}(m)=0$, for all $n\geq N$.  
We consider two cases: (a) Suppose $m\in \dom(f)$. Since $a_{(f,f_{n})}(m)=0$ and $b_{(f,f_{n})}(m)=0$,  for all $n\geq N$, we have that $m\in \dom(f_n)$ and $f_{n}(m)=f(m)$, for all $n\geq N$. (b) Suppose $m\not\in \dom(f)$. Since $a_{(f,f_{n})}(m)=0$, for all $n\geq N$, we have that $m\not\in \dom(f_{n})$, for all $n\geq N$. Then, by Theorem \ref{conv}, $f_n\stackrel{\tau_1}\to f$.

 For the other direction, suppose $g_n\stackrel{\tau_1}\to g$. Let $k\in \mathbb{N}$. Consider the following sets
 $$
 A=\{m\in \mathbb{N}\text{ }|\text{ }m\leq k\text{ and }m\in \dom(g)\}$$
 and
 $$
 B=\{m\in \mathbb{N}\text{ }|\text{ }m\leq k\text{ and }m\not\in \dom(g)\}.
 $$ 
 Since $g_n\stackrel{\tau_1}\to g$, there is $N\in \mathbb{N}$ such that if $n\geq N$, $g_n,g\in v(m,g(m))$, for all $m\in A$ and $g_n,g\in w_1(m)$, for all $m\in B$. Thus, $a_{(g,g_{n})}(m)=0$ and $b_{(g,g_{n})}(m)=0$ for all $m\leq k$ and $n\geq N$. Finally, we have that for all $n\geq N$
$$
\rho(g,g_{n}) = \displaystyle\sum_{t=1}^{\infty}\frac{a_{(g,g_{n})}(t)+b_{(g,g_{n})}(t)}{2^t}
=\displaystyle\sum_{t=k+1}^{\infty}\frac{a_{(g,g_{n})}(t)+b_{(g,g_{n})}(t)}{2^t}
\leq \displaystyle\sum_{t=k+1}^{\infty}\frac{1}{2^t}=\frac{1}{2^{k+1}}<\frac{1}{2^{k}}.
$$
Therefore $g_n\stackrel{\rho}\to g$.

\vspace{0.5cm}

(ii) It follows from (i) and Theorem  \ref{topI(X)}.
\end{proof}

\vspace{0.5cm}

The following fact shows that $2^\N$ is naturally embedded into $I(\N)$. Its easy proof is left to the reader. 

\begin{prop}
\label{metrica-cantor}
Let $\eta$ be given by 
$\eta(A,B)=\rho(1_A, 1_B)$, for $A,B\in\cantor$. Then $\eta$ is a compatible metric for $\cantor$.
\end{prop}

\begin{thm}
\begin{itemize}
\item[(i)] $(I(\N),\tau_{1})$ is a Polish semigroup and  $\rho$ is a complete compatible metric.

\item[(ii)] $(I(\N),\tau_{2})$ is a Polish semigroup and  $\rho^\ast$ is a complete compatible metric.
\end{itemize}
\end{thm}

\begin{proof} (i) By Proposition \ref{compa}, we only need to show that $\rho$ is complete. 
Let $(f_n)_n$ be a $\rho$-Cauchy sequence. Then $(\dom(f_n))_n$ is Cauchy in $\cantor$ (by Proposition \ref{metrica-cantor}), so it is convergent. Let $A=\displaystyle\lim_{n \to \infty}\dom(f_{n})$.  It is easy to verify that, for each $p\in A$, $(\ev_p(f_n))_n$ is Cauchy in $\N$ (with the discrete metric) and therefore eventually constant to a value $f(p)$.
Thus we have defined  a function $f$ such that $\dom(f)=A$  and  $f_n\stackrel{\rho}\to f$.
It is easy to see that $f$ is  injective. Finally, from  Theorem \ref{conv}, $f_n\stackrel{\tau_1}\to f$.

(ii) This follows from (i) and Proposition \ref{i_iso}. 
\end{proof}

From the previous result and Proposition \ref{28} we conclude that neither  $\rho$ nor $\rho^*$ are compatible with $\tau_{pp}$. As in Theorem \ref{metrica_suma}, we  define another metric as follows: 
\[
d(f,g)=\rho(f,g)+\rho^\ast(f,g).
\]

\begin{thm}
$(I(\mathbb{N}),\tau_{pp})$ is a Polish inverse semigroup and  $d$ is a complete compatible metric.
\end{thm}

\begin{proof} The proof of the compatibility of $d=\rho+\rho^*$ with $\tau_{pp}$ is similar to the proof of Theorem \ref{compa}. The completeness of $d$ follows from  Theorem \ref{metrica_suma} and Proposition \ref{compa}.
\end{proof}

\section{Open operation}

For $f\in I(X)$, let $r_f: I(X)\to R_f$ be given by $r_f(g)=g\circ f$ and 
\[
R_f=\{g\circ f\text{ }|\text{ }g\in I(X)\}.
\]
Define analogously $l_f$ and $L_f$ for the corresponding left operation. 

In contrast to what happens with topological groups, for a topological semigroup it is not true that $Vx$ is open when $V$ is open. Nevertheless,  we show that $r_f$ is an open map. We need this fact in the next section for proving a uniqueness result for $\tau_{pp}$. 

\begin{lema}
\label{oper-abierta}
Let $x,y,w,z\in X$. Then
\begin{enumerate}
\item $v(x,y)\circ v(z,x)=v(z,y)$.

\item $v(x,y)\circ v(z,w)= I(X)\setminus v(z,y)$ if $w\neq x$.
 
 \item $w_1(x)\circ w_1(y)= w_1(y)$. 
 
 \item $w_2(x)\circ w_2(y)=w_2(x)$.
 
\item $w_2(x)\circ w_1(y)=w_2(x)\cap w_1(y)$.

\item $w_1(x)\circ w_2(y)= I(X)$.

\item $w_1(y)\circ  v(x,y)=w_1(x)$.

\item $w_1(z)\circ  v(x,y)=I(X)$ if $z\neq y$.

\item  $v(x,y)\circ w_1(z)=w_1(z)$.

\item   $w_2(z)\circ  v(x,y)= [v(y,x)\circ w_1(z)]^{-1}=[w_1(z)]^{-1}=w_2(z)$.

\item $v(x,y)\circ  w_{2}(z)= [w_1(z)\circ v(y,x)]^{-1}$.
\end{enumerate}

\end{lema}

\begin{proof} All items are proved in an analogous way.

1. It's clear that $v(x,y)\circ v(z,x)\subseteq v(z,y)$. For the other inclusion, let $f\in v(z,y)$ and pick $A\subseteq X$ such that $x\in A$ and $|A|=|\dom(f)|$. Let  $h:\dom(f)\to A$  be any bijection such that $h(z)=x$ and let $g=f\circ h^{-1}$. Notice that $g:A\to\im(f)$ and $g(x)=y$, therefore $h\in v(z,x)$, $g\in v(x,y)$ and $f=g\circ h$.

\medskip

2. Suppose $w\neq x$. Let $f\in v(x,y)\circ v(z,w)$, then $f=g\circ h$, with $h\in v(z,w)$ and $g\in v(x,y)$. We consider two cases:

(a) Suppose $w\in \dom(g)$.  Then $g(w)\not=y$, since $w\not=x$. Notice that $f(z)=g(h(z))=g(w)\not=y$, therefore $f\in I(X)\setminus v(z,y)$.

(b) Suppose $w\not\in \dom(g)$.  We have that $h(z)=w\not\in \dom(g)$, therefore $z\not\in \dom(g\circ h)=\dom(f)$. Thus $f\in I(X)\setminus v(z,y)$.

Now let's see that $I(X)\setminus v(z,y)\subseteq v(x,y)\circ v(z,w) $. Let $f\in I(X)\setminus v(z,y)$. We consider four cases. 

\begin{enumerate}[(a)]
\item Suppose  $z\not\in \dom(f)$ and $y\in \im(f)$. Let $A\subseteq X$ be such that $w,x\in A$  and $|A|=|\dom(f)\cup \{z\}|$. Let  $h:\dom(f)\cup\{z\}\to A$  be any bijection such that $h(z)=w$, $h(f^{-1}(y))=x$ and let $g=f\circ h^{-1}$. Since $h(f^{-1}(y))=x$ we have that $g(x)=f(h^{-1}(x))=y$, therefore $h\in v(z,x)$, $g\in v(x,y)$ and $f=g\circ h$. 
    
\item Suppose $z\in \dom(f)$ and $y\not\in \im(f)$. Let $A\subseteq X$ be such that $w\in A$, $x\not\in A$  and $|A|=|\dom(f)|$. Let  $h:\dom(f)\to A$  be any bijection such that $h(z)=w$, and let $g:A\cup \{x\}\to \im(f)\cup\{y\}$ be such that $g(x)=y$ and  $g=f\circ h^{-1}$ in $A$. Notice that $h\in v(z,x)$, $g\in v(x,y)$ and $f=g\circ h$.
    
\item Suppose $z\in \dom(f)$ and $y\in \im(f)$.  Let $A\subseteq X$ be such that $w,x\in A$  and $|A|=|\dom(f)|$. Let  $h:\dom(f)\to A$  be any bijection such that $h(z)=w$, $h(f^{-1}(y))=x$ and let $g=f\circ h^{-1}$. Notice that $g:\im(f)\to A$ and $g(x)=y$, therefore $h\in v(z,x)$, $g\in v(x,y)$ and $f=g\circ h$.
    
\item Suppose $z\not\in \dom(f)$ and $y\not\in \im(f)$.  Let $A\subseteq X$ be such that $w\in A$, $x\not\in A$  and $|A|=|\dom(f)\cup \{z\}|$. Let  $h:\dom(f)\cup \{z\}\to A$  be any bijection such that $h(z)=w$, and let $g:A\cup \{x\}\to \im(f)\cup\{y\}$ be such that $g(x)=y$ and  $g=f\circ h^{-1}$ in $A$. Notice that $h\in v(z,x)$, $g\in v(x,y)$ and $f=g\circ h$.
\end{enumerate}

\medskip

3. It's obvious that $w_1(x)\circ w_1(y)\subseteq w_1(y)$. For the other inclusion, let $f\in w_1(y)$ and pick $A\subseteq X$ such that $x\not\in A$ and $|A|=|\dom(f)|$. Let  $h:\dom(f)\to A$  be any bijection  and let $g=f\circ h^{-1}$.  Notice that $g:\im(f)\to A$, and therefore $h\in w_1(y)$, $g\in w_1(x)$ and $f=g\circ h$. 

\medskip

4.  It follows from part 3., as  $[w_1(y)\circ w_1(x)]^{-1}=w_2(x)\circ w_2(y)$.

\medskip

5. It's easy to see that $w_2(x)\circ w_1(y)\subseteq w_2(x)\cap w_1(y)$. Let $f\in w_2(x)\cap w_1(y)$. Pick $A\subseteq X$  such that $|A|=|\dom(f)|$. Let  $h:\dom(f)\to A$  be any bijection  and let $g=f\circ h^{-1}$. Then  $h\in w_1(y)$, $g\in w_2(x)$ and $f=g\circ h$.

\medskip

6. Let $f\in I(X)$. Pick $A\subseteq X$ such that $x,y\not\in A$ and $|A|=|\dom(f)|$. Let  $h:\dom(f)\to A$  be any bijection  and let $g=f\circ h^{-1}$. Then  $h\in w_1(y)$, $g\in w_2(x)$ and $f=g\circ h$.

\medskip

7. It's easy to see that $w_1(y)\circ  v(x,y)\subseteq w_1(x)$.  Let $f\in w_1(x)$. Pick $A\subseteq X$ such that $y\not\in A$ and $|A|=|\dom(f)|$. Let  $h:\dom(f)\cup \{x\}\to A\cup\{y\}$  be any bijection such that $h(x)=y$, and let $g=f\circ h^{-1}$. Then $h\in v(x,y)$, $g\in w_1(y)$  and $f=h\circ g$.

\medskip

8. Let $f\in I(X)$. We have two cases.

(a) Suppose that $x\in \dom(f)$. Pick $A\subseteq X$ such that $y\in A$, $z\not\in A$  and $|A|=|\dom(f)|$. Let  $h:\dom(f)\to A$  be any bijection such that $h(x)=y$, and let $g=f\circ h^{-1}$. Then $h\in v(x,y)$, $g\in w_1(z)$  and $f=h\circ g$.

 (b) Suppose that $x\not\in \dom(f)$. Pick $A\subseteq X$  such that $y,z\not\in A$ and $|A|=|\dom(f)|$. Let  $h:\dom(f)\{x\}\to A\{y\}$  be any bijection such that $h(x)=y$, and let $g=f\circ h^{-1}$.  Then $h\in v(x,y)$, $g\in w_1(z)$  and $f=h\circ g$.
 
 \medskip
 
 9. Let us  see that $w_1(z)\subseteq v(x,y)\circ w_1(z)$. Let $f\in w_{1}(z)$ and suppose that $y\in \im(f)$. Let $A\subseteq X$ be such that $x\in A$ and $|A|=|\dom(f)|$. Let  $g:\dom(f)\to A$  be any bijection  such that $g(f^{-1}(y))=x$ and let $h=f\circ g^{-1}$. Then $f=h\circ g$, $h\in v(x,y)$.

Now, suppose that $f\in w_1(z)$ and $y\not \in \im(f)$. Let $A\subseteq X$ be  such that $x\not\in A$ and $|A|=|\dom(f)|$. Let $g:\dom(f)\to A$ be any bijective function. Let $h:A\cup \{x\}\to \im(f)\cup \{y\}$ be such that $h(x)=y$ and $(h\circ g)(v)=f(v)$ for $v\in \dom(f)$. We have that $h\in v(x,y)$, $g\in w_{1}(z)$ and $f=h\circ g$, that is, $f\in v(x,y)\circ w_1(z)$.

\medskip

Finally, $10.$ and $11.$ are evident. 

\end{proof}

\begin{lema}
\label{intersection}
Let  $A,B,C \subseteq I(X)$, then $A\circ(B\cap C)\subseteq (A\circ B)\cap (A\circ C)$.
\end{lema}

\proof Straightforward. \qed

\medskip

Now we can show that $r_f$ is an open map. 

\begin{thm}\label{right-open}
Let $f\in I(X)$. Then, $r_f:I(X)\to R_f$ is an open map  where $R_f$ is endowed with the relative topology to $\tau_{pp}$
\end{thm}

\begin{proof}
Let $U$ be a non empty basic $\tau_{pp}$-open set of the form: 
$$
U=\displaystyle\bigcap_{i=1}^nv(x_i,y_i)\cap\bigcap_{i=1}^mw_1(z_i)\cap\bigcap_{i=1}^pw_2(w_i).
$$ 
We will show that $U\circ f$ is open. Consider the sets:
\[
\begin{array}{rclcrcl}
  \widehat{M}   & = & \{x_i:\;1\leq i\leq n \} & \;\;\;\;\;\;\; & M  & = & \widehat{M}\cap \im(f) \\
   \widehat{N} & = & \{z_i:\;1\leq i\leq m\} & & N & = &  \widehat{N}\cap \im(f) \\
  \widehat{O} & = & \{w_i: \;1\leq i\leq p \}  & & O & = &  \widehat{O}\cap \im(f)  
\end{array}
\]
and 
$$
Q=\displaystyle\bigcap_{x_i\in M}v(f^{-1}(x_i),y_i)\cap\bigcap_{z_i\in N}w_1(f^{-1}(z_i))\cap\bigcap_{x_i\in \widehat{M}\setminus M}w_2(y_i)\cap\bigcap_{i=1}^pw_2(w_i).
$$
We claim that $U\circ f=Q\cap R_f$.
First, we show that $U\circ f\subseteq Q$. Let

\[
\begin{array}{rcl}
R     & = &  \displaystyle\bigcap_{x_i\in M}v(f^{-1}(x_i),x_i) \cap \bigcap_{z_i\in N}v(f^{-1}(z_i),z_i)\cap \bigcap_{w_i\in O}v(f^{-1}(w_i),w_i)\\
\\
S & = & \displaystyle\bigcap_{x_i\in \widehat{M}\setminus M}w_2(x_i)\cap \bigcap_{x_i\in \widehat{N}\setminus N}w_2(z_i)\cap \bigcap_{x_i\in \widehat{O}\setminus O}w_2(w_i).
\end{array}
\]
Then $f\in R\cap S$.  
By Lemma \ref{oper-abierta}, we have the following:
\medskip

\begin{itemize}
\item $v(x_i,y_i)\circ v(f^{-1}(x_i),x_i)=v(f^{-1}(x_i),y_i)$, for all $x_i\in M$.
    
\item $w_1(z_i)\circ v(f^{-1}(z_i),z_i)=w_1(f^{-1}(z_i))$, for all $z_i\in N$.
    
\item $w_2(w_i)\circ v(f^{-1}(w_i),w_i)=w_2(w_i)$, for all $w_i\in O$.  
    
\item $v(x_i,y_i)\circ w_2(x_i)=w_2(y_i)$, for all $x_i\in \widehat{M}\setminus M$.
    
\item $w_1(z_i)\circ w_2(z_i)=I(X)$, for all $z_i\in \widehat{N}\setminus N$.
    
\item $w_1(w_i)\circ w_2(w_i)=w_2(w_i)$, for all $w_i\in \widehat{O}\setminus O$.
\end{itemize}

\medskip

Therefore,  using Lemma \ref{intersection} and the claims above,  we easily have that 
\medskip

\[
\begin{array}{lcl}
U\circ f & \subseteq& U\circ(R\cap S)\\
\\
&\subseteq &  \displaystyle \bigcap_{x_i\in M}v(f^{-1}(x_i),y_i)\cap \bigcap_{z_i\in N}w_1(f^{-1}(z_i))\cap \bigcap_{w_i\in O}w_2(w_i)\cap \bigcap_{x_i\in \widehat{M}\setminus M}w_2(y_i)\cap \bigcap_{w_i\in \widehat{O}\setminus O}w_2(w_i)\\
     & =&  \displaystyle \bigcap_{x_i\in M}v(f^{-1}(x_i),y_i)\cap \bigcap_{x_i\in N}w_1(f^{-1}(z_i))\cap \bigcap_{i=1}^pw_2(w_i)\cap \bigcap_{x_i\in \widehat{M}\setminus M}w_2(y_i)\\
     &= & Q.
\end{array}
\]

\medskip

Now, we are going to show that $Q\cap R_f\subseteq U\circ f$.
Let $h\in Q\cap R_f$ and define $g=h\circ f^{-1}$. Notice  that $g(x_i)=y_i$, for all $x_i\in M$, since $h\in v(f^{-1}(x_i),y_i)$ for $x_i\in M$.
Notice also that that if $x_i\in \widehat{M}\setminus M$, then $x_i\not\in \dom(g)$, and $y_i\not\in \im(g)$ since $h\in w_2(y_i)$ for $x_i\in \widehat{M}\setminus M$. 
Therefore we can extend $g$ to a function $\widehat{g}$ such that $\dom(\widehat{g})$ is $\dom(g)\cup (\widehat{M}\setminus M)$  and  $g(x_i)=y_i$, for all $x_i\in \widehat{M}\setminus M$.

To finish the proof it suffices to show that 
$h=\widehat{g}\circ f$ and $\widehat{g}\in U$. 
First, we show that $\widehat{g}\in U$.
\begin{itemize}
\item $\widehat{g}(x_i)=y_i$, for all $i$, by construction. 
\item 
$z_i\not\in \dom(\widehat{g})$, for all $z_i$. In fact, if $z_i\in \widehat{N}\setminus N$, then $z_i\not\in \im(f)$, therefore $z_i\not\in \dom({g})$ and thus $z_i\not\in \dom(\widehat{g})$. On the other hand, if $z_i\in N$, then  $h\in w_1(f^{-1}(z_i))$ and 
$z_i\not\in \dom(\widehat{g})$.
\item As $h\in w_2(w_i)$ for all $w_i$, we have that $w_i\not\in \im(\widehat{g})$, for all $w_i$.   
\end{itemize}
Thus $\widehat{g}\in U$. Finally, since $g=h\circ f^{-1}$ and $\dom(\widehat{g})\cap \im(f)=\dom(g)$, 
$h=\widehat{g}\circ f$. 

\end{proof}

\begin{thm}
\label{left-open}
Let $f\in I(X)$ and consider $L_f=\{f\circ g\text{ }|\text{ }g\in I(X)\}$. Then, $\circ_f:I(X)\to L_f$, such that $\circ_f(g)=f\circ g$, is an open function, where $L_f$ is endowed with the relative topology to $\tau_{pp}$.
\end{thm}

\begin{proof}
Similar the proof of Theorem  \ref{right-open}.
\end{proof}

\medskip

We conjecture that $\circ$ is an open map. To show it would  require  analyzing many cases in order to  extend Lemma \ref{oper-abierta}. We did not pursue it further, since we do not need this result for this paper.

\section{Uniqueness of $\tau_{pp}$}

In this section we present some results showing that $\tau_{pp}$ is unique in some sense. 
For that end we first show that $(I(X),\tau_{pp})$ is  a quotient of a symmetric group $S_\infty(Y)$ (for some $Y$) with the usual product topology.

Let $X\subseteq Y$.   For each $f\in S_\infty(Y)$ we define $\widehat{f}\in I(X)$ as follows:
\[
\widehat{f}=\{(x,f(x)):\; x\in X \;\text{and}\; f(x)\in X\}.
\]
Thus $\widehat{f}$ is the restriction of $f$ to $f^{-1}(X)\cap X$.
Let $\pi:S_\infty(Y)\to I(X)$  be given by $\pi(f)=\widehat{f}$.

\begin{prop}
\label{cociente}
Let $X$ be an infinite set and $X\subseteq Y$. Then $\pi$ is onto if, and only if, $|X|\leq |Y\setminus X|$. 
\end{prop}

\begin{proof}
Suppose $|X|\leq |Y\setminus X|$ and let $g\in I(X)$. Let $A=\dom(g)$ and $B=\im(g)$. Let $C\subseteq Y\setminus X$ and $D\subseteq Y\setminus X$ be such that $|C|=|X\setminus A|$, $|D|=|X\setminus B|$, $|Y\setminus (X\cup C)|= |Y\setminus (X\cup D)|=|Y|$. All this conditions can be fulfilled  as $|X|\leq |Y\setminus X|$. Now take any extension of $g$ to a bijection $f: Y\to  Y$ such that $f[X\setminus A]=C$, $f[D]=X\setminus B$.  Then $\widehat{f}=g$.

Conversely, suppose $|X|> |Y\setminus X|$. Let $X=A\cup B$ be a partition of $X$ into sets of equal cardinality.  Then there is no $f\in S_\infty(Y)$ such that $\widehat{f}=1_A$. 
\end{proof}

\medskip

Let us observe that the map $\pi$ is not a semigroup  homomorphism. In fact, let $y\in Y\setminus X$, $x\in X$,  and $f\in S_\infty(Y)$ be such that $f(x)=y$,  $f(y)=x$ and $f(z)=z$ for all $z\not\in \{x,y\}$. Then $f\circ f=1_Y$, $\pi(1_Y)=1_X$ but $\pi(f)\circ \pi(f)=1_{X\setminus\{x\}}$.  In general, $\pi(f)\circ\pi(g)\subseteq \pi(f\circ g)$. 

\medskip

\begin{prop}
\label{cociente2}
Let $X\subseteq Y$ with $|X|\leq |Y\setminus X|$. The map $\pi: S_\infty(Y)\to (I(X),\tau_{pp})$ is continuous, onto  and open. 
\end{prop}

\begin{proof}
First, we will show that $\pi$ is continuous. Let $u(x,y)$  denote the subbasic open set of $S_\infty(Y)$ given by $\{f\in S_\infty(Y):\; f(x)=y\}$ for $x, y\in Y$. The continuity of $\pi$ follows from the following identities. 
\begin{itemize}
\item[(i)] $\pi^{-1}(v(x,y))=u(x,y)$ for all $x, y\in X$. 

\item[(ii)] $\pi^{-1}(w_1(x))=\bigcup_{y\in Y\setminus X} u(x,y)$ for $x\in X$.

\item[(iii)] $\pi^{-1}(w_2(y))=\bigcup_{x\in Y\setminus X} u(x,y)$ for $y\in X$.
\end{itemize}

To see that $\pi$ is open, let $x_i,y_i$ in $Y$ for $1\leq i\leq n$.
Let $1\leq k_1\leq k_2\leq k_3\leq n$ be such that  (1) $x_i,y_i\in X$ for $1\leq i\leq k_1$, (2)   $x_j\in X$ and $y_j\not\in X$ for $k_1<j\leq k_2$, (3)  $x_j\not\in X$ and $y_j\in X$ for $k_2<j\leq k_3$ and (4) $x_j\not\in X$ and $y_j\not\in X$ for $k_3<j\leq n$. Then 
\[
\pi( \bigcap_{i=1}^n u(x_i,y_i))=\bigcap_{i=1}^{k_1} v(x_i, y_i)\cap \bigcap_{j=k_1+ 1}^{k_2}w_1(x_j) \cap \bigcap_{j=k_2+ 1}^{k_3}w_2(y_j).
\]
The direction $\subseteq$ is straightforward. For  $\supseteq$ we use an analogous construction as in the proof that  $\pi$ is onto (Proposition \ref{cociente}).

\end{proof}

\begin{thm}
Let $X\subseteq Y$ with $|Y\setminus X|\geq |X|$. $\tau_{pp}$ is the only inverse semigroup Hausdorff topology on $I(X)$  with respect to which $\pi:S_\infty(Y)\to I(X)$ is continuous. 
\end{thm}

\begin{proof}
Let $\tau$ be an inverse semigroup Hausdorff  topology on $I(X)$ such that $\pi$ is continuous. Then by Theorem \ref{tppminima} we have that $\tau_{pp}\subseteq \tau$. Since $\pi$ is continuous,  by Proposition \ref{cociente2} we conclude that   $\tau\subseteq \tau_{pp}$.
\end{proof}

We now present a proof that $\tau_{pp}$ is the unique inverse semigroup Polish topology on $I(\N)$ (satisfying some addional conditions). Our approach is different than the one used in \cite{elliott2020}.  We need two auxiliary results.

\bigskip

\begin{lema}
\label{casi-convergencia}
Let $R_A=\{f\circ 1_A:\; f\in I(X)\}$  and $r_A:I(X)\to R_A$, $r_A(f)=  f\circ 1_A$. 
Let $\tau$ be an inverse semigroup Hausdorff topology on $I(X)$. Suppose $r_A$ is a $\tau$-open map for every $A\subseteq X$ cofinite. Let $(f_k)_k$ be a sequence on $I(X)$ such that 

\begin{itemize}
    \item[(i)] $f_k\stackrel{\tau_{pp}}{\longrightarrow} f$.
    \item[(ii)] There is a cofinite set $A\subseteq X$ such that $f_k\circ 1_A\stackrel{\tau}{\longrightarrow}f\circ 1_A$.
\end{itemize}
Then $f_k\stackrel{\tau}{\longrightarrow}f$.
\end{lema}

\begin{proof}
Let $A$ be as in (ii). Let $V\in\tau$ with $f\in V$. Since $\tau_{pp}\subseteq \tau$ (by Theorem \ref{tppminima}), we can assume that $V\subseteq v(x,f(x))\cap w_1(z)$ for all $x\in\dom(f)\cap (X\setminus A)$ and all $z\in (X\setminus A)\setminus \dom(f)$.  By hypothesis,  $V\circ 1_A$ is  open in $R_A$. By (ii), there is $k_0$ such that $f_k\circ 1_A\in V\circ 1_A$ for all $k\geq k_0$. Thus there is $g_k\in V$ such that $f_k\circ 1_A=g_k\circ 1_A$ for all $k\geq k_0$. By (i)  and theorems \ref{proyectiva} and \ref{conv}, there is $k_1\geq k_0$ such that $\dom(f_k)\cap (X\setminus A)= \dom(f)\cap (X\setminus A) $ and $f_k(x)=f(x)$ for all $x\in \dom(f)\cap (X\setminus A)$ and all $k\geq k_1$.  Since $g_k\in V$, $\dom(g_k)\cap (X\setminus A)= \dom(f)\cap (X\setminus A) $ and $g_k(x)=f(x)$ for all $x\in \dom(f) \cap (X\setminus A)$.  As $f_k\circ 1_A=g_k\circ 1_A$, then $f_k=g_k$ and hence $f_k\in V$ for all $k\geq k_1$.
\end{proof}

\begin{lema}
\label{unicidad1}
Let $\tau$ be a Polish inverse semigroup topology on $I(\N)$ such that $A\mapsto 1_A$ from $\cantor$ to $(I(\N),\tau)$ is continuous. Let $\N\subseteq Y$ with $Y\setminus\N$ infinite and $f_k\in S_\infty(Y)$, $k\in \N$ such that $f_k\to f$.  There is a dense $G_\delta$ set $G\subseteq S_\infty(Y)$ such that 
\[
\widehat{f_k}\circ \widehat{g}\stackrel{\tau}{\longrightarrow} \widehat{f}\circ \widehat{g}\;\; \mbox{for all $g\in G$}.
\]
\end{lema}

\begin{proof}
By Theorem \ref{tppminima} we have that $\tau_{pp}\subseteq \tau$. Since both topologies are Polish, by a well know classical result,  they have the same Borel sets (see \cite[Exercise 15.4]{KECH}). As $\pi$ is continuous with respect to $\tau_{pp}$,  we have that $\pi: S_\infty(Y)\to (I(\N), \tau)$ is Borel measurable.  Thus,  there is a dense $G_\delta$ set $H\subseteq S_\infty(Y)$ such that $\pi|_H: H\to (I(\N), \tau)$ is continuous (see \cite[Theorem 8.38]{KECH}). Let 
\[
L_{k+1}=\{g\in S_\infty(Y):\;  f_k\circ g\in H\},\;\;\; L_0=\{g\in S_\infty(Y):\; f\circ g\in H\}.
\]
Let $G=\bigcap_k L_k$. Since each $L_k$ is  dense $G_\delta$, so is $G$.   As $f_k\circ g\to f\circ g$ (in $S_\infty(Y)$) for all $g\in G$ and $\pi$ is continuous in $H$,  $\pi(f_k\circ g)\stackrel{\tau}{\longrightarrow} \pi(f\circ g)$ for all $g\in G$. 

As we said before, $\pi$ is not a homomorphism, however, we have the following 
\[
\pi(f_k)\circ \pi(g)= \pi(f_k\circ g)\circ 1_{\dom(\pi(f_k)\circ \pi(g))}.
\]
Notice that $\pi(f_k)\stackrel{\tau_{pp}}{\longrightarrow}\pi(f)$ as $\pi$ is continuous with respect to $\tau_{pp}$ (by Proposition \ref{cociente2}). As the function $\dom$ is continuous (by Theorem \ref{proyectiva}), we have 
\[
\dom(\pi(f_k)\circ \pi(g))\to\dom(\pi(f)\circ\pi(g)) 
\]
where the convergence is in the Cantor space $\cantor$.  Finally, by hypothesis  the map $A\mapsto 1_A$ is continuous with respect to  $\tau$, thus  we conclude 
\[
\pi(f_k) \circ\pi(g)\stackrel{\tau}{\longrightarrow} \pi(f)\circ \pi(g)
\]
for all $g\in G$.
\end{proof}

\begin{thm}
\label{unicidad2}
$\tau_{pp}$ is the unique inverse semigroup Polish topology on $I(\N)$ such that the collection of idempotent is compact and $r_A$ is a $\tau$-open map for every $A\subseteq \N$.
\end{thm}

\begin{proof}
By Theorem \ref{oper-abierta}, $r_A$ is a $\tau_{pp}$-open map and, by  Theorem \ref{proyectiva}, the collection of idempotent is $\tau_{pp}$-compact.
Conversely, let $\tau$ be a topology on $I(\N)$ as in the hypothesis. 
Let $\N\subseteq Y$ be such that $Y\setminus \N$ is countable. To have that $\tau=\tau_{pp}$,  it suffices to  show, by Lemma \ref{unicidad1},  that $\pi:S_\infty(Y)\to (I(\N),\tau)$ is continuous. 

Let $f_k\in S_\infty(Y)$ be a sequence converging to $f$. Let $G\subseteq S_\infty(Y)$ be a dense $G_\delta$ as in  Proposition \ref{unicidad1}. Since $\{h\in S_\infty(Y):\; \im(\widehat{h})\; \text{is cofinite}\}$ is dense $G_\delta$, let $g\in G$ be such that  $\im(\widehat{g})$ is cofinite.   Now notice that $\widehat{g}\circ \widehat{g}^{-1}= 1_{\im(\widehat{g})}$.
Let $A=\im(\widehat{g})$. Since $\widehat{f_k}\circ \widehat{g}\stackrel{\tau}{\longrightarrow} \widehat{f}\circ \widehat{g}$, by the continuity of $\circ$  we have 
$\widehat{f_k}\circ 1_A\stackrel{\tau}{\longrightarrow} \widehat{f}\circ 1_A$. Thus by Lemma \ref{casi-convergencia}, we conclude that $\widehat{f}_k\stackrel{\tau}{\longrightarrow} \widehat{f}$. 
\end{proof}

\bibliographystyle{plain}

\end{document}